\documentclass[11pt,reqno]{amsart}

\usepackage{amsmath,amssymb,mathrsfs,graphicx}
\usepackage[utf8]{inputenc}
\usepackage[T1]{fontenc}
\usepackage{amsmath,amsthm,amssymb,amscd}
\usepackage{amsfonts}
\usepackage{latexsym}
\usepackage{float}

\usepackage[foot]{amsaddr}

\usepackage{mathtools}      
\mathtoolsset{showonlyrefs} 

\usepackage{mathtools}
\usepackage{xargs}
\usepackage{graphicx}
\usepackage{caption,subcaption}
\usepackage{color,xcolor}
\usepackage{enumitem}

\usepackage{lmodern}

\usepackage[colorlinks=true,linkcolor=blue]{hyperref}

\usepackage{fullpage}   

\newtheorem{thm}{Theorem}

\newtheorem{prop}{Proposition}[section]
\newtheorem{lem}[prop]{Lemma}

\newtheorem{rmk}[prop]{Remark}

\usepackage{appendix}

\theoremstyle{definition}

\newtheorem{defn}[prop]{Definition}

\newcommand{\T}{\mathbb{T}}
\newcommand{\R}{\mathbb{R}}
\newcommand{\C}{\mathbb{C}}

\newcommand{\Z}{\mathbb{Z}}

\title{Intermittency of Riemann's non-differentiable function through the fourth-order flatness}

\author{Alexandre Boritchev$^1$}
\address{$^1$Universit\'e Claude Bernard -- Lyon 1, CNRS UMR 5208, Institut Camille Jordan, F-69622 Villeurbanne, France.}
\email{$^1$\tt alexandre.boritchev@gmail.com}

\author{Daniel Eceizabarrena$^2$}
\address{$^2$Department of Mathematics and Statistics, University of Massachusetts Amherst, Amherst MA 01003, United States.}
\email{$^2$\tt eceizabarrena@math.umass.edu}

\author{Victor Vila\c ca Da Rocha$^3$}
\address{$^3$School of Mathematics, Georgia Institute of Technology, Atlanta, Georgia 30332, United States.}
\email{$^3$\tt vrocha3@gatech.edu}


\begin{document}

\begin{abstract}
Riemann's non-differentiable function is one of the most famous examples of continuous but nowhere differentiable functions, but it has also been shown to be relevant from a physical point of view. Indeed, it satisfies the Frisch-Parisi multifractal formalism, which establishes a relationship with turbulence and implies some intermittent nature. It also plays a surprising role as a physical trajectory in the evolution of regular polygonal vortices that follow the binormal flow. With this motivation, we focus on one more classic tool to measure intermittency, namely the fourth-order flatness, and we refine the results that can be deduced from the multifractal analysis to show that it diverges logarithmically. 
We approach the problem in two ways: with structure functions in the physical space and with high-pass filters in the Fourier space.
\end{abstract}

\maketitle


\section{Introduction and motivation}
Riemann's non-differentiable function
\begin{equation}\label{OriginalRiemannFunction}
f(t) = \sum_{n=1}^{\infty}{\frac{\sin{ \left(n^2 t\right)}}{n^2}}, \quad t\in [0,2\pi],
\end{equation} 
is a celebrated example of a continuous but almost nowhere differentiable function. Weierstrass claimed \cite{Weierstrass} that it was introduced by Riemann in the 1860s, and since then it has been widely studied from an analytic perspective \cite{Duistermaat,Gerver1,Gerver2,Hardy,HolschneiderTchamitchian,Jaffard}. Intermittency, however, appears in the study of fully developed turbulence, closely related to multifractality. What is then the relationship between these two concepts which seem to be unrelated? 

On the one hand, there is a meeting point of the analytic study of Riemann's function and of the multifractal study of the velocity field of turbulent fluids, which is the Frisch-Parisi multifractal formalism \cite{FrischParisi}. It was designed as a heuristic formula to connect the regularity of the velocity of the fluid with the different scalings of its increments, but its mathematical range of validity was unclear and a program began in the 1990s with the objective of rigorously establishing it \cite{DaubechiesLagarias,Eyink,Jaffard, JaffardMultifractal,JaffardMultifractal2,JaffardOldFriends}. In particular, Jaffard \cite{Jaffard} gave the precise H\"older regularity of Riemann's function and proved that it satisfies the multifractal formalism, thus establishing an analytical relationship with turbulence.

One more connection is the unexpected appearance of Riemann's function in the dynamics of vortex filaments. In \cite{HozVega},  De la Hoz and Vega showed numerically that  the complex-valued version
\begin{equation}\label{RiemannFunction_Trajectory}
    \phi(t) = 2\pi i t + \sum_{n=1}^{\infty}{\frac{e^{2\pi i n^2 t}}{n^2}}, \quad t \in [0,1],
\end{equation}
appears as the trajectory of the corners of polygonal vortex filaments that follow the binormal flow. In Figure~\ref{Figure_Plot_Trajectory} we show the image of $\phi$ in the complex plane, while in  \cite{KumarVideos}  the reader can find a video of the numeric simulations, which captures the qualitative phenomenology of experiments done in \cite{ChicagoPeople}. For convenience, in  Figure~\ref{Figure_Evolution} we show a few snapshots of those numerical simulations. The trajectory in blue there and the image of the function in Figure~\ref{Figure_Plot_Trajectory} are surprisingly similar. 
Recently, Banica and Vega \cite{BanicaVega2020} proved the first rigorous result in this direction in a slightly different setting, where they consider a modification of the polygon.

\begin{figure}[h]
\begin{center}
\begin{subfigure}{0.45\textwidth}
  \centering
  \includegraphics[width=0.97\linewidth]{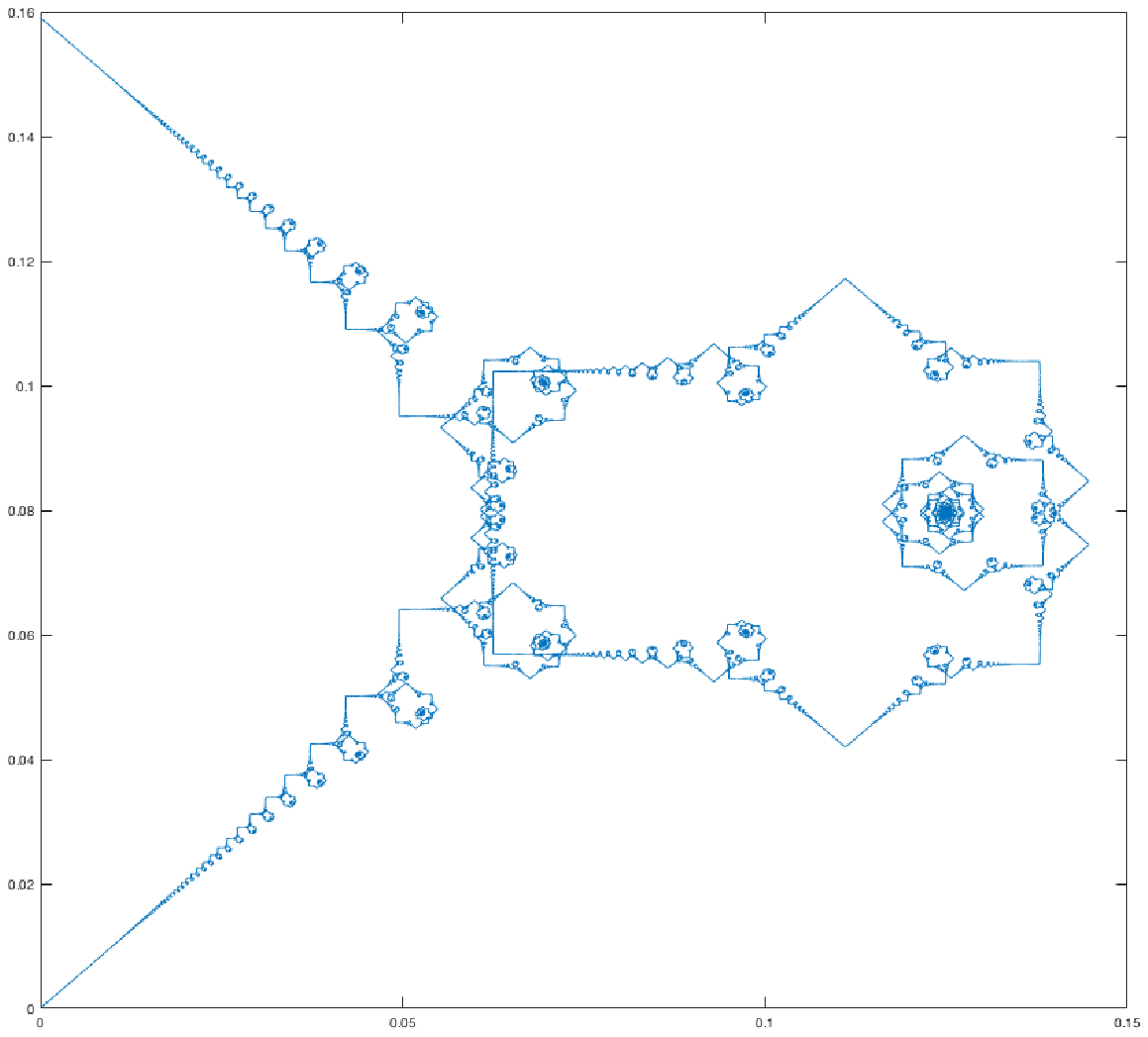}
  \caption{}
  \label{Figure_Plot_Trajectory}
\end{subfigure}%
\begin{subfigure}{0.03\textwidth}
  \centering
  \hspace{1pt}
\end{subfigure}%
\begin{subfigure}{0.45\textwidth}
  \centering
  \includegraphics[width=\linewidth]{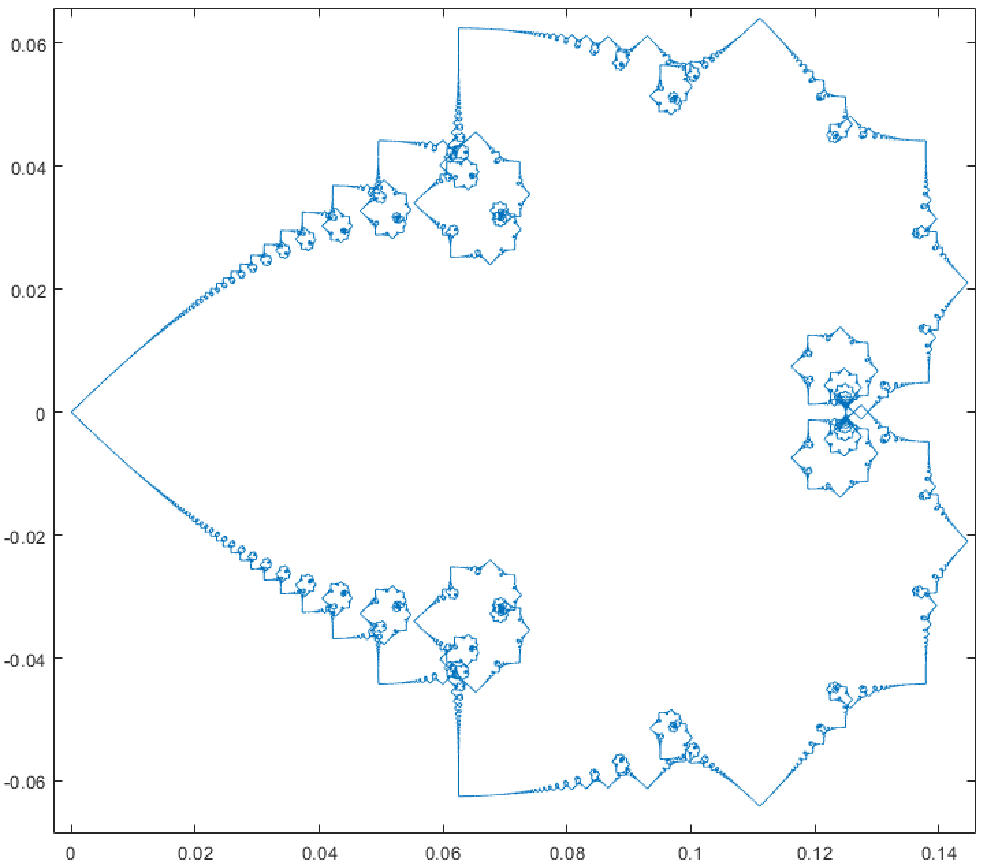}
  \caption{}
  \label{Figure_Plot}
\end{subfigure}
\end{center}
\label{Figure_RiemannPlots}
\caption{In Figure~\ref{Figure_Plot_Trajectory}, The image of $\phi$ defined in \eqref{RiemannFunction_Trajectory}. In Figure~\ref{Figure_Plot}, the image of $R$ defined in \eqref{RiemannFunction}.}
\end{figure}

\begin{figure}[h]
\centering
\begin{subfigure}{0.3\textwidth}
  \centering
  \includegraphics[width=1\linewidth]{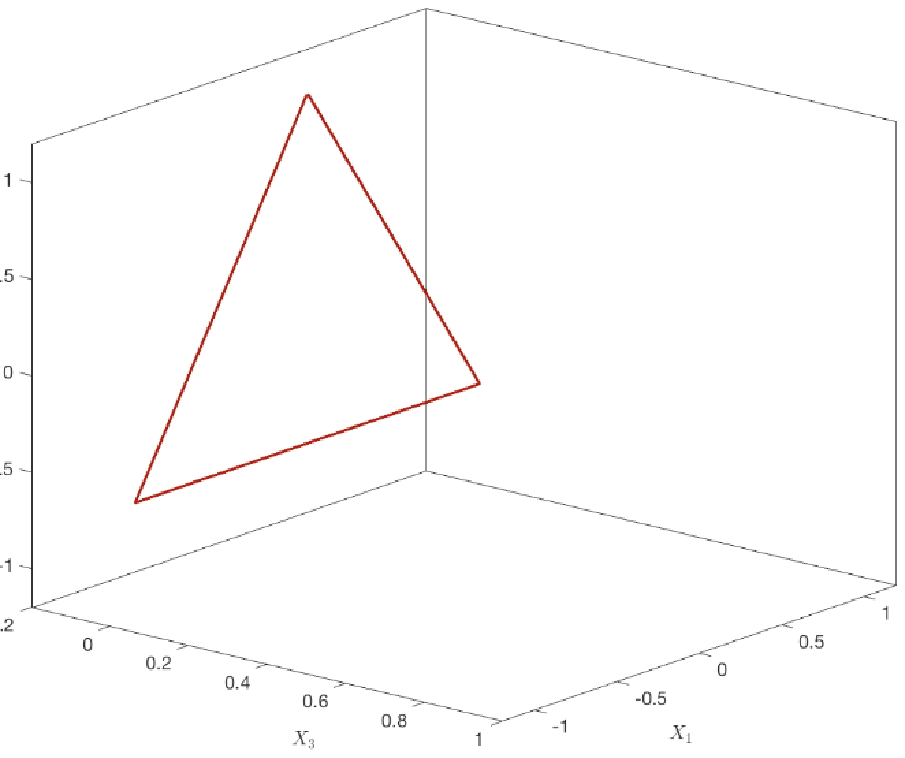}
  \caption{Fraction of the period : 0}
  \label{fig:Talbot0}
\end{subfigure}%
\begin{subfigure}{0.03\textwidth}
  \centering
  \hspace{1pt}
\end{subfigure}%
\begin{subfigure}{0.3\textwidth}
  \centering
  \includegraphics[width=1\linewidth]{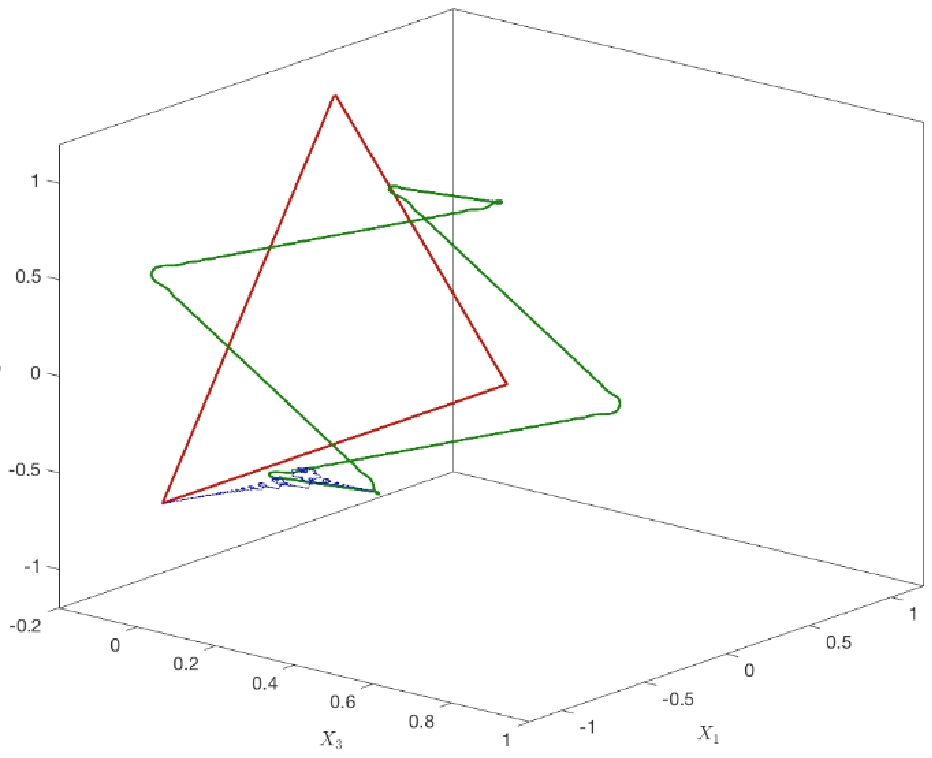}
  \caption{Fraction of the period : $\frac14$ }
  \label{fig:Talbot14}
\end{subfigure}
\begin{subfigure}{0.03\textwidth}
  \centering
  \hspace{1pt}
\end{subfigure}%
\begin{subfigure}{0.3\textwidth}
  \centering
  \includegraphics[width=1\linewidth]{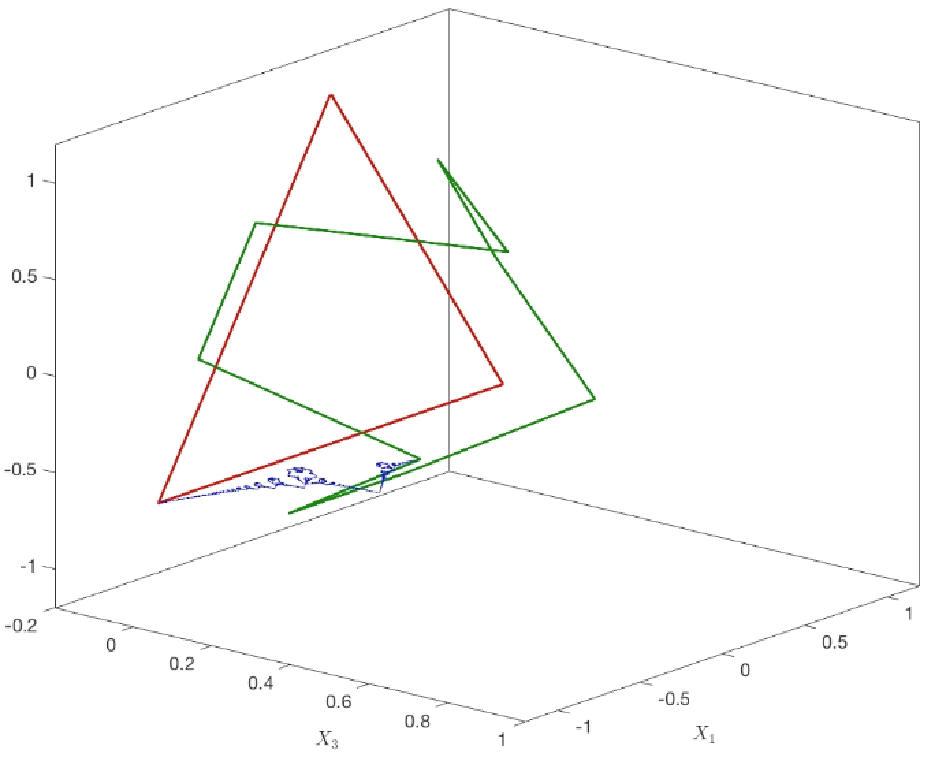}
  \caption{Fraction of the period : $\frac13$}
  \label{fig:Talbot13}
\end{subfigure}
\begin{subfigure}{0.3\textwidth}
  \centering
  \includegraphics[width=1\linewidth]{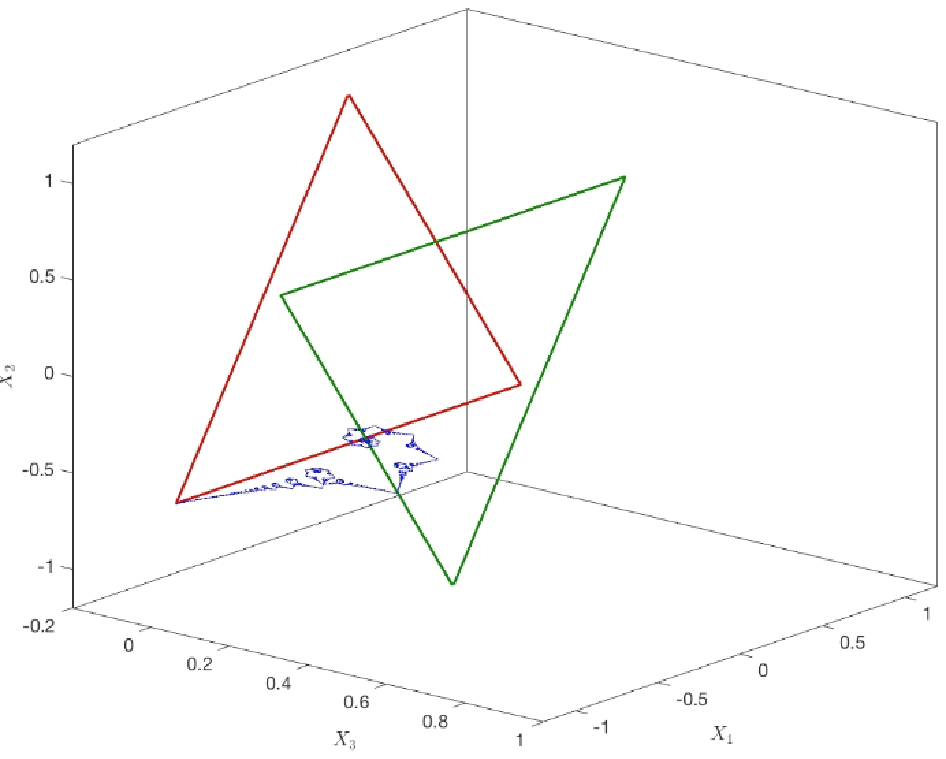}
  \caption{Fraction of the period : $\frac12$}
  \label{fig:Talbot12}
\end{subfigure}
\begin{subfigure}{0.03\textwidth}
  \centering
  \hspace{1pt}
\end{subfigure}%
\begin{subfigure}{0.3\textwidth}
  \centering
  \includegraphics[width=1\linewidth]{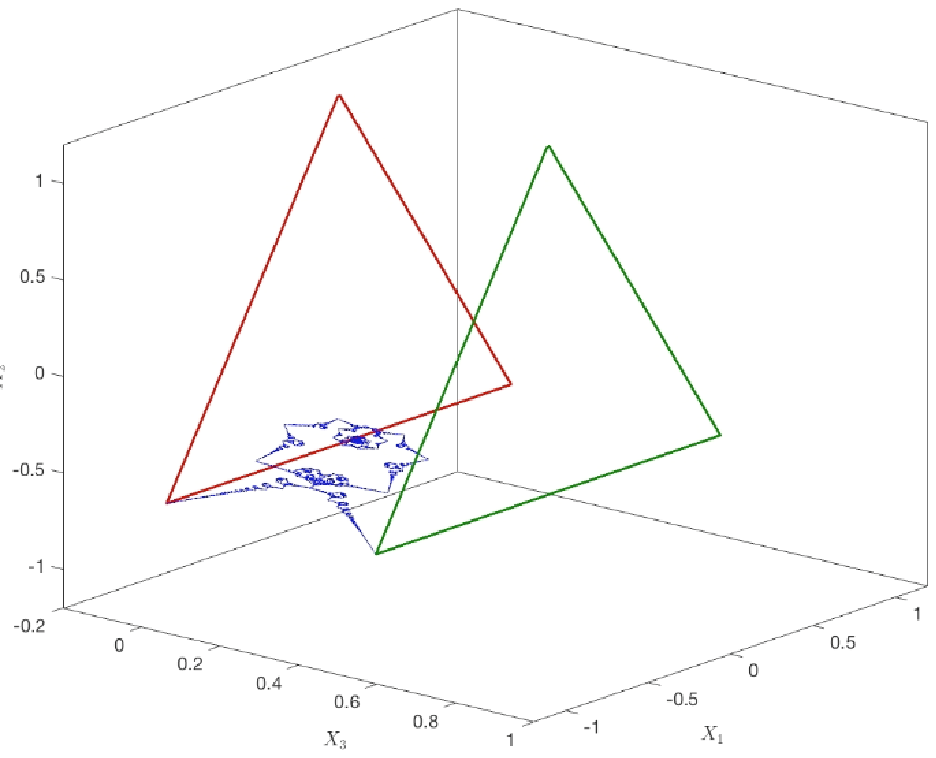}
  \caption{Fraction of the period : 1}
  \label{fig:Talbot1}
\end{subfigure}
\begin{subfigure}{0.03\textwidth}
  \centering
  \hspace{1pt}
\end{subfigure}%
\caption{
Snapshots of the numeric simulation in \url{https://youtu.be/leKT9deeZYA} by Sandeep Kumar \cite{KumarVideos}.
 They show the evolution of a triangular vortex filament, which is periodic, at fractions 0, 1/4, 1/3, 1/2 and 1 of the period. In blue, the trajectory of a corner. 
}
\label{Figure_Evolution}
\end{figure}

The two connections given above make it natural to study properties of $\phi$ that are physically motivated and related to turbulence. From the analytic point of view, it is usual to work with 
\begin{equation}\label{RiemannFunction}
R(t) = \sum_{n=1}^{\infty}{\frac{e^{2\pi i n^2 t}}{n^2}}, \quad t \in [0,1],
\end{equation}
which is the immediate adaptation of the original Riemann's function to the complex plane. Its image is shown in Figure~\ref{Figure_Plot}. Observe that $R$ and $\phi$ share analytic properties because their difference is smooth.

The objective of our work is to give quantitative estimates of the intermittency of $R$. In particular, we will identify a logarithmic correction that cannot be deduced from Jaffard's multifractal result \cite{Jaffard}. The main conclusion of this paper, stated here qualitatively, is the following:
\begin{thm}\label{TheoremIntroduction}
In the sense of the fourth-order flatness, Riemann's non-differentiable function defined in \eqref{RiemannFunction} is intermittent.
\end{thm} 
We will introduce and motivate all concepts involved in this theorem in Section~\ref{SECTION_PhysicalBackground}, and we will also write the corresponding quantitative statement in Theorem~\ref{Theorem_Quantitative}. We will rigorously define all concepts in  Section~\ref{SECTION_Statement_Of_The_Result} and give the exact quantitative formulation in Theorem~\ref{Thm1}.

Let us briefly explain the structure of the document. Before precisely stating the results, in  Subsection~\ref{SUBSECTION_Multifractality}  we elaborate on the physical background that motivates this work, that is,  the relationship between turbulence and intermittency, the Frisch-Parisi multifractal formalism and Jaffard's work on Riemann's function, and motivate the description of intermittency through the flatness.
For the sake of completeness of the physical background, in Subsection~\ref{SUBSECTION_VFE} we explain with more detail how Riemann's function appears in the evolution of polygonal vortex filaments that follow the binormal flow.
Then, following the multifractal literature and the ideas of Frisch \cite[Chapter 8]{Frisch}, in Section~\ref{SECTION_Statement_Of_The_Result} we set notation, we adapt the definition of intermittency to functions by means of high-pass filters and structure functions, and we state the main Theorem~\ref{Thm1}. 
In Sections~\ref{SECTION_HighPassFilters} and \ref{SECTION_StructureFunctions} we prove Theorem~\ref{Thm1}. 
Finally, in Appendixes~\ref{APPENDIX_LittlewoodPaley} and \ref{APPENDIX_Zalcwasser} we include some auxiliary results that we need in the proof.

\section{Physical Background}\label{SECTION_PhysicalBackground}

\subsection{Turbulence, intermittency, multifractality and Riemann's function}\label{SUBSECTION_Multifractality}

In a series of papers in 1941, Kolmogorov gave a precise statistical description of the properties of turbulent flows.
These articles, usually referred to as the K41 theory, are seminal for the theory of turbulence. 
The idea is that away from the boundaries and at scales neither too large nor too small
(i.e., in the inertial range), the behaviour of a turbulent flow is universal: 
it does not depend on the fluid under consideration, nor on the geometry of the particular mechanisms producing it. The K41 theory assumes that the flows are 
homogeneous (statistically invariant under translations) and isotropic (invariant under rotations), and that the velocity increments are statistically self-similar.

\begin{rmk}\label{Remark_Dissipation}
One more assumption of K41 is the existence of a dissipation scale, below which the velocity is damped by the effect of viscosity. 
However, many results in this theory assume that the Reynolds number tends to infinity, which can be thought of the viscosity tending to zero. This implies that the dissipation scale goes to zero. In this situation, the inertial range covers scales that are arbitrarily small.
\end{rmk}

Although this theory gives remarkable results, such as the Kolmogorov four-fifths law \cite{Kolmo}, 
two main issues have been pointed out: the lack of universality, emphasized by Landau in 1944 and reported later \cite{Landau} (it was just a footnote in the first version of the book), and the lack of self-similarity of the velocity increments in the inertial range,  highlighted by several experiments. Indeed, the velocity of a fluid in fully developed turbulence may erratically change even over very small distances. This lack of self-similarity is called \textit{intermittency}.

Several ways to model intermittency can be found in the literature. One method, discussed by Frisch \cite[Sections 8.2 and 8.3]{Frisch}, is to measure the variability of the velocity in small scales by means of the increments
\begin{equation}\label{VelocityIncrements}
\delta v(r,l)  = v(r + l) - v(r), \qquad r,l \in \mathbb{R}^3
\end{equation}
from a statistical point of view. For that purpose, structure functions 
\begin{equation}\label{FrischStructureFunction}
 S_p(\ell)=\left< \left|\delta v(r,l)\right|^p \right>
\end{equation}
are used, which are very relevant in the literature of turbulence (see \cite{Frisch, Jaffard} and also \cite{Boritchev2013, Boritchev2014, Boritchev2016} for rigorous results for the Burgers equation).
Here, $p \geq 1$ and $\ell = |l| >0$, and $\langle \cdot \rangle$ denotes an ensemble average, defined as the mean over many realizations of the flow with different initial conditions and forcing. 
It could as well be a temporal average if we assume the ergodicity of the flow, or a spatial one if we use its homogeneity. In the study of turbulence, these three definitions are usually considered to be equivalent.
In any case, $S_p(\ell)$ does not depend on $r$ by homogeneity, and depends only on $\ell = |l|$ by isotropy. 

Structure functions can be combined with the probabilistic concept of kurtosis or flatness, which measures the \textit{tailedness}, that is, the thickness of the tails of a probability density function. Shortly, the larger the flatness, the higher the probability of getting outlier values. Since one way to account for a high variability of the velocity is to see that it often takes values far from the mean, it is reasonable to measure it in terms of the flatness
\begin{equation}\label{FrischFlatStructure}
 G(\ell)=\frac{S_4(\ell)}{S_2(\ell)^2}=\frac{\left< |\delta v(\ell)|^4 \right>}{\left< |\delta v(\ell)|^2 \right>^2}.
\end{equation}
Following Remark~\ref{Remark_Dissipation}, we will say that the flow is intermittent if $G(\ell)$ tends to infinity when the scale parameter $\ell$ tends to zero. In terms of self-similarity, the key point is that $G(\ell)$ is constant for flows with self-similar velocity increments.
Indeed, if the increments $\delta v(r,\ell)$ are self-similar with scaling exponent $h$,  then $\delta v(r,\lambda\ell) \stackrel{\text{Prob.}}{=} \lambda^h\delta v(r,\ell)$ for any $r$ and for any $\lambda>0$, in the sense that both $\delta v(r,\lambda\ell)$ and $\lambda^h\delta v(r,\ell)$ have the same probability distribution.
This implies that $G(\lambda\ell)=G(\ell)$ for all $ \lambda>0$.  
Since, by the definition above, a flow that is intermittent cannot have a constant flatness, it cannot have self-similar increments. This matches the heuristic definition of intermittency as the lack of self-similarity given in the beginning of this section. 

An alternative idea, also discussed in \cite{Frisch}, is to work with high-pass filters of the velocity in the Fourier space. This consists in removing the lower Fourier modes, so that only the oscillations with largest frequencies remain. If the velocity of the flow $v$ can be expressed in terms of its Fourier transform
\begin{equation*}
 v(r)=\int_{\R^3}\hat v(\omega)\,e^{i\omega \cdot r}\, \mathrm d\omega, 
\end{equation*}
then the high-pass filtered velocity $v_{>\Omega}$ is defined by 
\begin{equation*}
 v_{>\Omega}(r)=\int_{|\omega|>\Omega}\hat v(\omega) \, e^{i\omega \cdot r}\, \mathrm d\omega, \qquad \quad \text{ for } \Omega >0.
\end{equation*}
Since the highest frequencies measure the small-scale behaviour, the same idea of adapting the kurtosis to the statistical distribution of the velocity suggests defining the flatness alternatively by
\begin{equation}\label{FrischFlatFilters}
 F(\Omega)=\frac{\left\langle \left|v_{>\Omega}(r)\right|^4 \right>}{\left< \left|v_{>\Omega}(r)\right|^2 \right>^2},
\end{equation}
which does not depend on $r$ because of homogeneity. We say that the flow is intermittent
if $F(\Omega)$ tends to infinity when $\Omega$ grows.
As for $G(\ell)$, $F(\Omega)$ is constant for self-similar flows.
In this situation, there is an exponent $h >0$ such that for all $r$ and all $\lambda >0$ we have $v_{>\lambda \Omega}(r) \stackrel{\text{Prob.}}{=} \lambda^{-h}\, v_{>\Omega}(r)$, that is, $v_{>\lambda \Omega}(r)$ and $\lambda^{-h}\, v_{>\Omega}(r)$ have the same probability distribution. As a consequence, $F(\lambda\Omega)=F(\Omega)$ for every $\lambda>0$. 
Therefore, the more $F(\Omega)$ grows with $\Omega$, the less $v$ is self-similar at small scales, 
and thus by definition the more intermittent the flow $v$ is.

\begin{rmk}\label{rmkHPfilters}
Frisch \cite[Section 8.2]{Frisch} defines the high-pass filters of the velocity $v$ with respect to time instead of space. In this case, he suggests that  the inverse of the flatness measures the fraction of the time when the studied signal is ``on''. 
This remark leads to a more intuitive definition of intermittency: 
a signal is said to be intermittent if it displays activity during only a fraction of time that decreases with the scale of time under consideration.
\end{rmk}

Very related to the above is the Frisch-Parisi multifractal formalism. 
If the velocity $v$ is highly variable even over short distances, then one may expect that even for small variations of $r$, the scaling for the velocity increments
\begin{equation}\label{DifferentScalings}
|\delta v(r,l) | \approx |l|^\alpha, \qquad r,l \in \mathbb{R}^3,
\end{equation}
will hold for very different values of $\alpha$. It is generally accepted that the relevant quantity is the dimension $d(\alpha)$ of the set of points where \eqref{DifferentScalings} holds for a fixed $\alpha$, the so-called spectrum of singularities of the velocity. It is expected that many non trivial values will exist for different $\alpha$. However, since it is very difficult to measure that directly in experiments, structure functions \eqref{FrischStructureFunction} are used alternatively. The contribution of Frisch and Parisi \cite{FrischParisi} was that under the assumption that the structure functions \eqref{FrischStructureFunction} satisfy $S_p(\ell) \approx \ell^{\zeta(p)}$ when $\ell$ is small, we can expect that
\begin{equation}\label{FrischParisiConjecture}
d(\alpha) = \inf_{p} \left( \alpha p - \zeta(p) + 3 \right).
\end{equation}
This formula is known as the Frisch-Parisi conjecture or the multifractal formalism. 

But why would Riemann's non-differentiable function be related to any of these concepts? Let us begin by saying that the argument of Frisch and Parisi leading to their multifractal formalism \eqref{FrischParisiConjecture} is completely heuristic, and hence, its mathematical validity is doubtful. However, the formalism can be rigorously adapted to the setting of measures and functions. Let us focus on the case of functions $f$ of a real variable. For $\alpha>0$, $f$ is said to be locally $\alpha$-H\"older regular at a point $x_0$, and denoted $f \in C^{\alpha}(x_0) $,
if there exists a polynomial $P_{x_0}$ of degree at most $\lfloor \alpha \rfloor$ such that  $\left| f( x_0 + h) - P_{x_0}(h) \right| 
\leq C |h|^{\alpha}$ for small enough $h\in\mathbb{R}$. Let the H\"older exponent of $f$ at $x_0$ be 
$$
H_f(x_0)=\sup\ \left\{\alpha:\ f \in C^{\alpha}(x_0) \right\}.
$$
Then, the spectrum of singularities is defined for each $\alpha >0$ as the Hausdorff dimension of the set of points $x_0$ having H\"older exponent $\alpha$, 
\[ d_f(\alpha) = \dim_{\mathcal{H}}\{ x_0 : H_f(x_0) = \alpha \},  \]
where by convention the dimension is $-\infty$ in case the set is empty.  Besides, it is natural to adapt the definition of structure functions in \eqref{FrischStructureFunction} as
 \begin{equation}\label{StructureFunctionsJaffard}
S_{f,p}(\ell) = \int{|f(x+\ell)-f(x)|^p\,\mathrm dx}.
\end{equation}
Then, if $\zeta_f(p)$ is the exponent that best fits the behaviour $S_{f,p}(\ell) \approx \ell^{\zeta_f(p)}$, in the sense that $\zeta_f(p) = \liminf_{\ell \to 0} \log S_{f,p}(\ell) / \log \ell$, then the multifractal formalism asserts that 
\begin{equation}\label{FrischParisiConjectureFunctions}
d_f(\alpha) = \inf_{p} \left( \alpha p - \zeta_f(p) + 1 \right).
\end{equation}
The change from the 3 in \eqref{FrischParisiConjecture} to 1 is due to the velocity $v$ being three dimensional, while $f$ here is a function in $\mathbb{R}$. If working with functions in $\mathbb R^d$, the 3 in \eqref{FrischParisiConjecture} should be replaced by $d$.

There is, however, no reason by which \eqref{FrischParisiConjectureFunctions} should be true for all $f$, and it became an interesting mathematical problem to decide for which functions it does hold.
Some partial results were given concerning functions in Sobolev spaces \cite{Jaffard1992}, and extended to Besov spaces \cite{Eyink}.
Also, the formalism was checked for several \textit{toy} examples \cite{DaubechiesLagarias} and for generic and self-similar functions \cite{JaffardMultifractal,JaffardMultifractal2}. Multifractality of some classical functions  was studied as well \cite{JaffardOldFriends}. 
In such works, precise definitions of the exponent $\zeta(p)$ were proposed using several Sobolev-type spaces and also wavelet techniques and wavelet leaders (see  \cite[Sections 2 and 3]{Jaffard2004}). Recently, the genericity of the multifractal formalism was shown in $\mathbb R^d$ in the sense that every concave, continuous and compactly supported function with maximum equal to $d$ is the spectrum of singularities of a whole category of functions that satisfy the multifractal formalism \cite{BarralSeuret2020}. 

But most importantly for us, Jaffard \cite{Jaffard} proved that Riemann's function $R$ satisfies the multifractal formalism \eqref{FrischParisiConjectureFunctions} by computing the exponents
\begin{equation}\label{ExponentJaffard}
    \zeta(p) = \left\{ \begin{array}{cc}
        3p/4 & \text{ when } p \leq 4,  \\
        1 + p/2 & \text{ when } p > 4,
    \end{array}
    \right.
\end{equation}
with the definition $\zeta(p) = \sup\{ s \mid f \in B^{s/p,\infty}_p  \}$ where $B^{s/p,\infty}_p $ are Besov spaces, and also 
\begin{equation*}
d(\alpha) = \left\{ \begin{array}{ll}
4\alpha - 2, & \alpha \in [1/2,3/4], \\
0, & \alpha = 3/2, \\
-\infty, & \text{otherwise,}
\end{array} \right.
\end{equation*}
thus extending the known results on its analytic regularity.
Therefore, a relationship between this traditionally analytic object and turbulence was established.

This result by Jaffard motivates the study of intermittency from the point of view of the flatness as discussed earlier. Indeed, even if the functional approach is enough to compute $\zeta(p)$, there is still some ambiguity in the behavior of the structure functions. If $S_p(\ell)$ follows a power law, then it should be $\ell^{\zeta(p)}$, but one cannot rule out corrections of lower order like logarithms. These corrections are critical for the flatness and for the definition we have for intermittency, since the power laws alone would imply that  $G(\ell) = S_4(\ell)/S_2(\ell)^2 \approx \ell^3 / (\ell^{3/2})^2 = 1$, while a logarithmic correction in $S_4(\ell)$ could make $G(\ell)$ grow logarithmically. In this paper, we determine the precise behavior of $S_4(\ell)$ and $S_2(\ell)$.
\begin{thm}\label{Theorem_Quantitative}
For Riemann's non-differentiable function defined in \eqref{RiemannFunction},
\begin{equation}
    S_4(\ell) \simeq \ell^3 \log(\ell^{-1}), \qquad \qquad S_2(\ell) \simeq \ell^{3/2}, \qquad \text{ for } \ell \ll 1,
\end{equation}
so $G(\ell) \simeq \log(\ell^{-1})$. Analogue results hold for high-pass filters. Thus, Riemann's non-differentiable function is intermittent. 
\end{thm}

This logarithmic correction suggests that $p=4$ plays a special role for Riemann's function. Observe that it also coincides with the change of behavior for $\zeta(p)$ in \eqref{ExponentJaffard}.

In Section~\ref{SECTION_Statement_Of_The_Result} we will give precise definitions for all concepts discussed here, and the rigorous and precise statement of Theorem~\ref{Theorem_Quantitative} can be found in Theorem~\ref{Thm1}, both for structure functions and high-pass filters.

\subsection{The binormal flow, the vortex filament equation and Riemann's function}\label{SUBSECTION_VFE}
 
As we mentioned in the introduction, there is an astonishing connection of Riemann's non-differentiable function and the evolution of vortex filaments following the binormal flow \cite{HozVega}. This flow, governed by the vortex filament equation
\begin{equation}\label{VFE}
\boldsymbol{X}_t=\boldsymbol{X}_x\wedge \boldsymbol{X}_{xx}, \qquad (t,x)\in \R\times\R,
\end{equation}
is a model for one-vortex filament dynamics. In this equation, $\boldsymbol{X} = \boldsymbol{X}(t,x)\in\R^3$ is a curve parametrized by arclength $x$ and time $t$ and $\wedge$ is the usual cross product. It is easy to see that \eqref{VFE} can equivalently be written as $\boldsymbol{X}_t = \kappa\, \boldsymbol{B}$, where $\kappa$ is the curvature and $\boldsymbol{B}$ is the binormal vector, hence its name. 
 
A remarkable result about this equation was given by Hasimoto \cite{Hasimoto}, who proved that the transformation
\begin{equation}\label{HasimotoTransformation}
\Psi(t,x) = \kappa(t,x)\, e^{i\, \int_0^x{\tau(t,\sigma)\,\mathrm d\sigma}},
\end{equation}  
$\tau$ being the torsion, solves the nonlinear Schr\"odinger equation
\begin{equation}\label{NLS}
i\Psi_t+\Psi_{xx}+\frac12\left(|\Psi|^2+A(t)\right)\Psi=0,
\end{equation}
where $A$ is a real, time dependent function that depends on $\kappa, \tau$ and their derivatives \cite{BanicaVega2012}. 
This way, the Hasimoto transformation supplies a method to find solutions of \eqref{VFE},
since they can be produced from particular solutions to \eqref{NLS} \cite{BanicaVega2012}.

Data with corners of different shapes have been considered in \cite{GutierrezRivasVega,BanicaVega2012,BanicaVega2015,BanicaVega2018}, which serve as a model for  filaments of air in a delta wing during a flight \cite{HozGarciaCerveraVega}, as well as the corners that are created after the reconnection of two different filaments in the rear of a plane, or even in the study of superfluid helium \cite{Schwarz}. These data are important also from an analytic point of view, since they correspond to the self-similar solutions to the equation.

We are interested in similar situations where, for a given $M \in \mathbb{N}$, the initial datum is a closed, regular and planar $M$-sided polygon. This can be seen as a superposition of the corners coming from the self-similar solutions \cite{HozVega2018}, and it is a model for experiments with  smoke rings produced from polygonal-like nozzles \cite{ChicagoPeople}. 
By planar we mean that $\tau(0,x) = 0$, so the initial datum  in \eqref{NLS} is $\Psi(0,x) = \kappa(0,x)$. An option to parametrise the initial curvature is to place $M$ equidistributed Dirac deltas in the interval $[0,2\pi)$ and then to extend this interval periodically to the real line, so that we have 
\begin{equation}\label{InitialDatum}
\Psi_M(0,x) = \frac{2\pi}{M}\,\sum_{k \in \mathbb{Z}}\delta\left(x - 2\pi \frac kM\right), \qquad x \in \mathbb{R}. 
\end{equation} 
Then, the Galilean invariance of \eqref{NLS} can be used to determine
\begin{equation}\label{GalileanSolution}
\Psi_M(t,x) = \widehat{\Psi_M}(t,0)\, \sum_{k \in \mathbb{Z}}e^{ - i(Mk)^2t + iMkx},
\end{equation}
and a very nice connection with the optical Talbot effect \cite{BerryMarzoliSchleich,Talbot,Rayleigh} is brought to light by showing that at every scaled rational time $t_{p,q} = (2\pi/M^2)(p/q)$ the curve $\boldsymbol{X}_M(t_{p,q},x)$ is again a polygon, which now is not necessarily planar but has $Mq$ sides if $q$ is odd, and $Mq/2$ sides if $q$ is even (see Figure~\ref{Figure_Evolution} or \cite[Section 5.3]{JerrardSmets} for some numerical simulations).

Also in \cite{HozVega}, the temporal trajectories of the corners of the initial filament, represented by $\boldsymbol{X}_M(t,2\pi k/M)$ for every fixed $k \in \mathbb{Z}$, are considered. 
Due to the translation invariance of \eqref{NLS} and the periodicity of the initial datum, all such trajectories are the same. It was shown numerically that, up to scaling, $\boldsymbol{X}_M(t,0)$ is extremely similar to the image of the function 
\begin{equation}\label{RiemannPhi}
\phi(t) = \sum_{k \in \mathbb{Z}}{ \frac{e^{2\pi i k^2 t} - 1}{ k^2} } = 2\pi i t + 2\, \sum_{k = 1}^\infty{ \frac{e^{2\pi i k^2 t} - 1}{ k^2} } + \frac{\pi^2}{3}, 
\end{equation}
which is essentially \eqref{RiemannFunction_Trajectory}. If $\widehat{\Psi_M}(t,0) = 1$, then $\phi$ is the integral of $\Psi_M(\cdot, 0)$. Heuristically, this integration can be seen as the analogue of unmaking Hasimoto's transformation \eqref{HasimotoTransformation}.  
Strictly speaking, a constant $\widehat{\Psi_M}(t,0)$ corresponds to the free Schr\"odinger equation but, also heuristically, it works for \eqref{NLS} by adjusting the function $A(t)$.

Moreover, the similarity between $\boldsymbol{X}_M(t,0)$ and Riemann's function improves when $M$ increases and it is indistinguishable to the eye with $M=10$ (see \cite[Figure 3]{HozVega}). This has recently been proved partially in \cite{BanicaVega2020}, further motivating the study of Riemann's function from both a geometric and physical point of view.
Geometric results for $R$ were obtained \cite{Duistermaat,ChamizoCordoba}, while the Hausdorff dimension and tangency properties of the image of $\phi$ were obtained by the second author \cite{Eceizabarrena0,Eceizabarrena1,Eceizabarrena2}. On the other hand, the study of the intermittency of Riemann's function in this paper follows the physical approach. We remark that the setting of the vortex filament equation explained here matches the setting of Kolmogorov's theory in Subsection~\ref{SUBSECTION_Multifractality} and specially in Remark~\ref{Remark_Dissipation}; indeed, this equation is derived from the Euler equation, which is itself a simplification of the Navier-Stokes equation with zero viscosity. Thus, in the analysis of the vortex filament equation and Riemann's function we expect no dissipation range, and phenomenology corresponding to the inertial range should be observed for scales as small as we want.

\section{Statement of the result}\label{SECTION_Statement_Of_The_Result}

\subsection{Setting and notation.}\label{SUBSECT_Notation}

Let $\T=\R/\Z$ be the circle. 
For $p \geq 1$, we denote by $L^p$ the Lebesgue space on the circle, $L^p(\mathbb{T})$, and by $\ell^p$ the Lebesgue sequence space $\ell^p(\mathbb{Z})$. We work with functions such that the corresponding Fourier series 
\begin{equation}\label{fourierseries}
 f(x)=\sum_{n\in\Z}a_ne_n(x)
\end{equation}
are absolutely convergent, where $\left(e_n\right)_{n\in\Z}$ is the orthonormal basis of $L^2(\mathbb{T})$ defined by $e_n(x)=e^{2\pi in x}$ 
and $a_n\in\C$ are the Fourier coefficients of $f$. In particular, this implies that $f$ is continuous, and therefore $f\in L^p(\mathbb{T})$ for every $p \in \left[ 1,+\infty \right]$. 

In the case of Riemann's non-differentiable function, we use the notation 
\[ R(x) = \sum_{n=1}^{\infty}{ \frac{e^{2\pi i n^2 x}}{ n^2 } } = \sum_{k=1}^{\infty}{ \frac{\sigma_k}{k} e^{2\pi i k x} }, \]
where $\sigma_k$ is defined by
\begin{equation}\label{IndicatorSquares}
\sigma_k =
\begin{cases}
	1, & \text{ if } k \text{ is the square of an integer}, \\
	0, & \text{ otherwise.	}
\end{cases}
\end{equation}

For two positive functions $f$ and $g$, we write $f\lesssim g$ to denote that there exists a constant $C>0$ such 
that $f\leq Cg$. We also write $f\simeq g$ to denote that $f\lesssim g$ and $g\lesssim f$. 
If the constants involved depend on some parameter $\alpha$, we write $f \lesssim_\alpha g$ and $f \simeq_\alpha g$.

For any $a,b \in \mathbb{R}$ such that $a < b$, we write $\sum_{n=a}^{b} = \sum_{n \in [a,b] \cap \mathbb{Z}}$ and  $\sum_{n>a}^{b} = \sum_{n \in (a,b] \cap \mathbb{Z}}$.

\subsection{Flatness and intermittency}\label{SUBSECT_Flatnesses_And_Main_Results}

In the deterministic setting of Riemann's non-differentiable function, as already suggested in Subsection~\ref{SUBSECTION_Multifractality}, definitions \eqref{FrischFlatFilters} and \eqref{FrischFlatStructure} need to be modified. The standard way to do so is to substitute $p$-moments by the $p$-th powers of $L^p$ norms, as was done when going from \eqref{FrischStructureFunction} to \eqref{StructureFunctionsJaffard}.

\begin{defn}\label{DefHP}
For $f:\mathbb{T} \to \mathbb{C}$ and $N \in \mathbb{N}$, we define the \textbf{high-pass filter} and the \textbf{low-pass filter} as the projections of $f$ on Fourier modes above $N$ and below $N$, respectively, defined by
\begin{equation}\label{HighPassFilterDefinition}
f_{\geq N}(x) = \sum_{|n| \geq N}{a_n e_n(x)} \qquad \text{ and } \qquad f_{\leq N}(x) = \sum_{|n| \leq N}{a_n e_n(x)},
\end{equation}
where $a_n$ are the Fourier coefficients of $f$. Filters $f_{> N}$ and $f_{< N}$ with strict inequalities are defined analogously. 
The \textbf{flatness} of $f$ in the sense of high-pass filtering is given by
	\begin{equation}\label{Flatness}
	F_f(N) = \frac{ \lVert f_{\geq N} \rVert^4_{L^4(\mathbb{T})} }{ \lVert f_{\geq N} \rVert^4_{L^2(\mathbb{T})} }.		
	\end{equation}
According to what we said after \eqref{FrischFlatFilters} in Subsection~\ref{SUBSECTION_Multifractality}, we say that $f$ is \textbf{intermittent in the sense of high-pass filtering} if
\[ \lim_{N \to +\infty}{ F_f(N)} = +\infty. \]
\end{defn}

\begin{defn}\label{DefSF}
Let $p \geq 1$, $f:\T\to\C$ a bounded and measurable function and $\ell \in [0,1]$. The \textbf{structure functions} of $f$ are defined by
\begin{equation}\label{StructureFunctions}
S_{f,p}(\ell) = \int_{\T}{|f(x+\ell)-f(x)|^p\,\mathrm dx}.
\end{equation} 
The \textbf{flatness} of $f$ in the sense of structure functions is
	\begin{equation}\label{StructureFlatness}
	G_f(\ell) = \frac{ S_{f,4}(\ell) }{ S_{f,2}(\ell)^2 }.		
	\end{equation}
Again, according to \eqref{FrischFlatStructure} in Subsection~\ref{SUBSECTION_Multifractality}, we say that $f$ is \textbf{intermittent in the sense of structure functions} if
\[ \lim_{\ell \to 0}{ G_f(\ell)} = +\infty. \]
\end{defn}

\begin{rmk}
If there is no risk of confusion regarding $f$, we write $S_p(\ell)$ instead of $S_{f,p}(\ell)$. 
\end{rmk}

\subsection{Main result}\label{SUBSECTION_MainResult}

The following theorem, the rigorous version of Theorem~\ref{TheoremIntroduction}, is the main result of the paper.
\begin{thm}\label{Thm1}
Let $R$ be Riemann's non-differentiable function \eqref{RiemannFunction}. 
There exist $N_0 \in \mathbb{N}$ and $0 < \ell_0 < 1$ such that for $N > N_0$ and $\ell < \ell_0$, we have 
\begin{equation}
    \lVert R_{\geq N} \rVert_4^4 \simeq N^{-3} \, \log N, \qquad \qquad \lVert R_{\geq N} \rVert_2^2 \simeq N^{-3/2}.
\end{equation}
and
\begin{equation}
    S_{R,4}(\ell) \simeq \ell^3\, \log (\ell^{-1} ), \qquad \qquad S_{R,2}(\ell) \simeq \ell^{3/2}.
\end{equation} 
These fourth order logarithmic corrections imply
\[ F_R(N)    \simeq \log N, \qquad \qquad G_R(\ell) \simeq \log (\ell^{-1}), \]
so $R$ is intermittent in the sense of both high-pass filtering and structure functions.
\end{thm}

\begin{rmk}
Both \eqref{HighPassFilterDefinition} and \eqref{StructureFunctions} capture the small-scale behavior of $R$. Indeed, the high-pass filter in $N$ implies working with oscillations smaller than $N^{-1}$, while with structure functions we directly measure differences in scale $\ell$. Thus, there is a natural identification of the small-scale parameters $N^{-1}$ and $\ell$, so $F_R(N)$ and $G_R(\ell^{-1})$ should measure the same phenomenon. The result in Theorem~\ref{Thm1} is consistent with this fact.
\end{rmk}

\subsection{Discussion}

Let us discuss the intermittency of Riemann's function from the point of view of its graph, as well as from the perspective of the evolution of vortex filaments.

The set $R([0,1])$ is not self-similar, but the asymptotic behaviour of $R$ in \cite{Duistermaat} and also Figure~\ref{Figure_Plot} reveal at least the presence of some approximate self-similar structure. Therefore, if we understand intermittency as a measure of the lack of self-similarity, $R$ should have somewhat weak intermittent properties, and its flatness should show this. The logarithmic growth of both $F_R$ and $G_R$ in Theorem~\ref{Thm1} agrees with this interpretation.

Regarding the evolution of polygonal vortex filaments, our result is related to a couple of interesting open questions:
\begin{itemize}
 
    \item Supported by numerical evidence \cite{HozVega}, the natural conjecture is that Riemann's function is the limit of the trajectories of the corners when the number of sides of the polygon tends to infinity, which has only been proved for a modified version of the polygons \cite{BanicaVega2020}. By the turbulent nature of vortex filaments, it is reasonable to expect that the trajectory of the corners is intermittent. Thus, the proof that Riemann's function is intermittent further supports the conjecture.

   \item In addition to determining if, for a fixed $M\in \mathbb N$, the trajectory of the corners of the $M$-sided polygon is intermittent, one more question is whether it satisfies the multifractal formalism.  
    
\end{itemize}

\section{Intermittency in the sense of high-pass filters}\label{SECTION_HighPassFilters}

To prove the part of Theorem~\ref{Thm1} concerning high-pass filters, we will use the Littlewood-Paley decomposition of $R$, as well as a result of Zalcwasser \cite{Zalcwasser} on the $L^4$ norm of the sum of square-phased exponentials. Both results are stated in Appendixes~\ref{APPENDIX_LittlewoodPaley} and \ref{APPENDIX_Zalcwasser}.

We estimate the $L^2$ norm of the high-pass filter first.
\begin{lem}\label{LemmaL2}
For every $N \geq 2$, 
\[ \lVert R_{\geq N} \rVert_{L^2(\mathbb{T})} \simeq N^{-3/4}. \]
\end{lem}
\begin{proof}
By Plancherel's theorem, we get
 \[  \lVert R_{\geq N} \rVert_{L^2(\mathbb{T})}^2 
= \int_0^1{ \left| \sum_{n=N}^{\infty}\frac{\sigma_n}{n}\,e^{2\pi i n x} \right|^2\,\mathrm dx } 
= \sum_{n=N}^{\infty} \frac{\sigma_n}{n^2} 
= \sum_{n=\sqrt{N}}^\infty \frac{1}{n^4} \simeq \int_{\sqrt{N}}^\infty \frac{\mathrm dx}{x^4} = \frac{1}{N^{3/2}}.   \]
\end{proof}

To compute the $L^4$-norm of the high-pass filter, one may try to use Plancherel's theorem again, since $\lVert f \rVert_4^4 = \lVert f^2 \rVert_2^2$ holds. However, 
\begin{equation}\label{NumberTheory}
R_{\geq N}^2(x) = \sum_{k=2N}^\infty \left( \sum_{n=N}^{k-N} \frac{\sigma_n\,\sigma_{k-n}}{n\,(k-n)} \right)\,e^{2\pi i k x}, 
\end{equation}  
whose Fourier coefficients are related to $\sum_{n=N}^{k-N} \sigma_n\,\sigma_{k-n}$, the number of ways in which $k$ can be written as a sum of two squares both of which are greater than $N$. The study of such sums is a classical problem in number theory and can be very technical. Instead, the Fourier series of $R$ can be decomposed in frequency pieces that act almost independently by the Littlewood-Paley decomposition. We use this technique to prove the following lemma.
\begin{lem}\label{LemmaL4LowerBound}
There exists $N_0 \in \mathbb{N}$ such that 
\[  \lVert R_{\geq N} \rVert_{L^4(\mathbb{T})} \gtrsim N^{-3/4}\, (\log{N})^{1/4}, \qquad \forall N \geq N_0. \]
\end{lem}

\begin{proof}
According to Appendix~\ref{APPENDIX_LittlewoodPaley},
the Littlewood-Paley decomposition of $R_{\geq N}$ is
\begin{equation}\label{LittlewoodPaleyDecomposition}
R_{\geq N}(x) = \sum_{j=1}^{\infty}\Delta_jR_{\geq N}(x), 
\end{equation} 
where the Littlewood-Paley pieces are
\[  \Delta_0R_{\geq N}(x) = \sum_{1 \leq n < A}\frac{\sigma_n}{n}\,e^{2\pi i n x} \qquad \text{ and } 
\qquad \Delta_jR_{\geq N}(x) = \sum_{A^j \leq n < A^{j+1}}\frac{\sigma_n}{n}\,e^{2\pi i n x}, \qquad \forall j \in \mathbb{N}.   \]
The value of $A$ will be chosen later (see \eqref{ChoiceOfA}).
We define $j(N)$ as the index corresponding to the piece containing the $N$-th Fourier coefficient, the only one satisfying $A^{j(N)} \leq N < A^{j(N)+1}$. Then, $\Delta_j R_{\geq N} = 0$ for every $j < j(N)$. By the Littlewood-Paley theorem,
 we may write the inequality
\begin{equation}\label{LittlewoodPaleyBoundBelow}
\lVert R_{\geq N} \rVert_{L^4(\mathbb{T})} \simeq \left\lVert \left( \sum_{j \geq j(N)} \left|  \Delta_jR_{\geq N}\right|^2 \right)^{1/2}  \right\rVert_{L^4(\mathbb{T})} \geq \lVert \Delta_{i(N)}R \rVert_{L^4(\mathbb{T})}, 
\end{equation} 
where $i(N) = j(N)+1$, using that $\Delta_{i(N)}R = \Delta_{i(N)}R_{\geq N}$. The choice of $i(N)$ comes from the fact that it is the first complete Littewood-Paley piece after $j(N)$, which is truncated as a consequence of the high-pass filter. 

Let us estimate $\lVert \Delta_{i(N)}R \rVert_{L^4(\mathbb{T})}$. As in \eqref{NumberTheory}, we use Plancherel's theorem to write
\begin{equation}\label{PlancherelInLittlewoodPaley}
\lVert \Delta_{i(N)}R \rVert_{L^4}^4 =  \lVert (\Delta_{i(N)}R)^2 \rVert_{L^2}^2 
= \left\lVert \left( \sum_{n = A^{i(N)}}^{A^{i(N)+1}} \frac{\sigma_n}{n} \,e^{2\pi i n x} \right)^2  \right\rVert_{L^2}^2  
\simeq \sum_{k=2A^{i(N)}}^{2A^{i(N)+1}} \left| \sum_{n} \frac{\sigma_n\,\sigma_{k-n}}{n\,(k-n)} \right|^2, 
\end{equation}  
where the index $n$ must satisfy $A^{i(N)} \leq n \leq A^{i(N)+1}$ and $A^{i(N)} \leq k- n \leq A^{i(N)+1}$. In both cases, $n \simeq A^{i(N)}$ and $k-n \simeq A^{i(N)}$. Hence, we can take the denominators outside the sum:
\begin{equation}\label{RemovingDenominator}
 \lVert \Delta_{i(N)}R \rVert_{L^4}^4  
\simeq \frac{1}{A^{4i(N)}} \, \sum_{k=2A^{i(N)}}^{2A^{i(N)+1}} \left| \sum_{n} \sigma_n\,\sigma_{k-n} \right|^2 
= \frac{1}{A^{4i(N)}} \,  \left\lVert \sum_{n = A^{i(N)}}^{A^{i(N)+1}} \sigma_n\,e^{2\pi i n x}  \right\rVert_{L^4}^4.  
\end{equation}   
This is a sum of exponentials with squared phases, whose $L^p$ norms were computed by Zalcwasser \cite{Zalcwasser}. First, the triangle inequality gives 
\begin{equation}\label{TriangleInequality}
\left\lVert \sum_{n = A^{i(N)/2}}^{A^{(i(N)+1)/2}} e^{2\pi i n^2 x}  \right\rVert_{L^4}  
\lesssim  \left\lVert \sum_{n = 1}^{A^{(i(N)+1)/2}} e^{2\pi i n^2 x}  \right\rVert_{L^4} + \left\lVert \sum_{n = 1}^{A^{i(N)/2}} e^{2\pi i n^2 x}  \right\rVert_{L^4},
\end{equation}
so by Zalcwasser's theorem in Appendix~\ref{APPENDIX_Zalcwasser} we get, for large enough $N$,
\begin{equation}\label{UpperBoundForPiece}
\begin{split}
\lVert \Delta_{i(N)}R \rVert_{L^4} & \lesssim \frac{1}{A^{i(N)}}\, \left( \left( A^{i(N)+1}\log{A^{(i(N)+1)/2}}  \right)^{1/4}  + \left( A^{i(N)}\log{A^{i(N)/2}} \right)^{1/4} \right) \\
& \lesssim \frac{1}{A^{i(N)}} \, \left(A^{i(N)}\,\log A^{i(N)/2}\right)^{1/4}  \lesssim A^{-3i(N)/4}\, \left( \log A^{i(N)}\right)^{1/4} .
\end{split}
\end{equation}
On the other hand, using the reverse triangle inequality in \eqref{TriangleInequality}, we get
\begin{equation}\label{InverseTriangleInequality}
\left\lVert \sum_{n = A^{i(N)/2}}^{A^{(i(N)+1)/2}} e^{2\pi i n^2 x}  \right\rVert_{L^4}  
\gtrsim   \left\lVert \sum_{n = 1}^{A^{(i(N)+1)/2}} e^{2\pi i n^2 x}  \right\rVert_{L^4} - \left\lVert \sum_{n = 1}^{A^{i(N)/2}} e^{2\pi i n^2 x}  \right\rVert_{L^4} .
\end{equation}   
Let us denote the constants in Zalcwasser's theorem for $p=4$
by $0 < c < C$. We get
\begin{equation}\label{LowerBoundForPiece}
\begin{split}
\lVert \Delta_{i(N)}R \rVert_{L^4}  & \gtrsim \frac{1}{A^{i(N)}}\, \left(  \left( c\,A^{i(N)+1}\log{A^{(i(N)+1)/2}}  \right)^{1/4}  -  \left( C\, A^{i(N)}\log{A^{i(N)/2}} \right)^{1/4} \right) \\
& = \frac{1}{A^{i(N)}}\, (A^{i(N)}\,\log{A^{i(N)/2}})^{1/4}\, \left( (c\, A(1 + i(N)^{-1}))^{1/4} - C^{1/4} \right) \\
& \gtrsim  A^{-3i(N)/4}\, \left(\log A^{i(N)} \right)^{1/4}  \, \left( c^{1/4}A^{1/4}-C^{1/4} \right).
\end{split}
\end{equation}    
Finally, choose $A$ so that 
\begin{equation}\label{ChoiceOfA}
(cA)^{1/4}-C^{1/4} = 1.
\end{equation}  
Observe that in the proof of \eqref{UpperBoundForPiece} and \eqref{LowerBoundForPiece} we may replace $i(N)$ by any $j \geq i(N)$, so we have proved that 
\begin{equation}\label{EstimateForLittlewoodPaleyPieces}
\lVert \Delta_jR \rVert_{L^4} \simeq A^{-3j/4}\,\left(\log A^j  \right)^{1/4}, \qquad \forall j \geq i(N).
\end{equation}
Coming back to $j=i(N)$, from \eqref{LittlewoodPaleyBoundBelow}, since $A^{i(N)-1} \leq N < A^{i(N)}$, we get
\[  \lVert R_{\geq N} \rVert_{L^4(\mathbb{T})} \gtrsim  N^{-3/4}\,\left(\log N \right)^{1/4}.  \]
\end{proof}

Lemma~\ref{LemmaL4LowerBound} suffices to prove that the flatness of Riemann's non-differentiable function tends to infinity. However, we can be more precise and show that the lower bound in the lemma is sharp. 
\begin{lem}\label{LemmaL4UpperBound}
There exists $N_0 \in \mathbb{N}$ such that 
\[  \lVert R_{\geq N} \rVert_{L^4(\mathbb{T})} \lesssim N^{-3/4}\, (\log{N})^{1/4}, \qquad \forall N \geq N_0. \]
\end{lem}
\begin{proof}
Applying the triangle inequality in the Littlewood-Paley decomposition \eqref{LittlewoodPaleyDecomposition}, we write
\begin{equation}\label{LittlewoodPaleyBoundAbove}
\lVert R_{\geq N} \rVert_{L^4} \leq \lVert \Delta_{j(N)}R_{\geq N} \rVert_{L^4} 
+ \sum_{j \geq i(N)} \lVert \Delta_j R \rVert_{L^4}.
\end{equation}
Using \eqref{EstimateForLittlewoodPaleyPieces} we can estimate $\lVert \Delta_j R \rVert_{L^4}$ for $j \geq i(N)$. To deal with the index $j(N)$, following the arguments in
  \eqref{PlancherelInLittlewoodPaley}, \eqref{RemovingDenominator}, \eqref{TriangleInequality} and \eqref{UpperBoundForPiece} 
and using $A^{j(N)} \leq N < A^{j(N)+1}$, we write
\begin{equation}\label{EstimateForJPiece}
\begin{split}
 \lVert \Delta_{j(N)}R_{\geq N} \rVert_{L^4} &  
\simeq \frac{1}{N}\, \left\lVert \sum_{n = N}^{A^{j(N)+1}} \sigma_n\,e^{2\pi i n x}  \right\rVert_{L^4} 
\lesssim \frac{1}{N}\, \left( \left\rVert \sum_{n = 1}^{A^{j(N)+1}} \sigma_n\,e^{2\pi i n x}  \right\rVert_{L^4} 
+  \left\rVert  \sum_{n = 1}^N \sigma_n\,e^{2\pi i n x}  \right\rVert_{L^4} \right) \\
 & \lesssim \frac{1}{N}\, \left( N\, \log N \right)^{1/4} = N^{-3/4}\,(\log N)^{1/4}.
\end{split}
\end{equation}  
On the other hand, using \eqref{EstimateForLittlewoodPaleyPieces}, we bound
\begin{equation}\label{Tail}
\sum_{j \geq i(N)} \lVert \Delta_j R \rVert_{L^4} \lesssim \sum_{j \geq i(N)} A^{-3j/4}\,\left(\log A^j  \right)^{1/4} \simeq \sum_{j \geq i(N)} j^{1/4}\, A^{-3j/4}.
\end{equation}
By H\"older's inequality one can write 
\begin{equation}\label{HolderInequality}
 \sum_{j \geq i(N)} j^{1/4}\, A^{-3j/4} \leq \left( \sum_{j \geq i(N)} j\, A^{-3j/4} \right)^{1/4}\, \left( \sum_{j \geq i(N)} A^{-3j/4}\right)^{3/4} . 
\end{equation} 
The second sum is geometric and equals
\[ \frac{A^{-3i(N)/4}}{1-A^{-3/4}} \simeq A^{-3i(N)/4}. \]
The first one can be computed differentiating power series. Indeed, for $|r| < 1$ we write 
\begin{equation}
\begin{split}
 \sum_{j \geq i(N)} j\, r^j & \leq \sum_{j \geq i(N)} (j+1)\, r^j = \frac{d}{dr}\, \sum_{j \geq i(N)} r^{j+1}  = \frac{d}{dr}\, \frac{r^{ i(N)  + 1}}{1-r} \\
 & =  \frac{ i(N)  \, r^{ i(N) } }{1-r} \left(  1 + \frac{1}{i(N)} \left(1 + \frac{r}{1-r} \right) \right) \\
 & \leq \frac{2}{1-r} \, i(N) \, r^{i(N)}.
\end{split}
\end{equation}  
The last inequality is satisfied when $i(N) > 1 + r/(1-r)$, which holds when $N > A^{1+r/(1-r)}$.
Choosing $r = A^{-3/4}$, from \eqref{Tail} and \eqref{HolderInequality}, we get
\[   \sum_{j \geq i(N)} \lVert \Delta_j R \rVert_{L^4} \lesssim i(N)^{1/4}\, A^{-3i(N)/4} \simeq N^{-3/4}\, (\log N)^{1/4} .  \]
Combining this with 
 \eqref{LittlewoodPaleyBoundAbove} and \eqref{EstimateForJPiece}, 
we finally obtain
\[ \lVert R_{\geq N} \rVert_{L^4} \lesssim N^{-3/4}\, (\log N)^{1/4}, \]
for large enough $N$.
\end{proof}

From Lemmas~\ref{LemmaL2}, \ref{LemmaL4LowerBound} and \ref{LemmaL4UpperBound}, the proof of the high-pass filter part of Theorem~\ref{Thm1} is immediate.

\begin{proof}[Proof of Theorem~\ref{Thm1} (Part 1)]
For large enough $N$ determined by Lemmas~\ref{LemmaL4LowerBound} and \ref{LemmaL4UpperBound}, we may write
\[ F_R(N) = \frac{ \lVert R_{\geq N} \rVert_{L^4(\mathbb{T})}^4 }{ \lVert R_{\geq N} \rVert_{L^2(\mathbb{T})}^4 } \simeq \frac{ N^{-3}\, \log N   }{  N^{-3} } = \log N, \]
and therefore $\displaystyle\lim_{N \to \infty}{F_R(N)} = +\infty$. 
\end{proof}

\section{Intermittency in the sense of structure functions}\label{SECTION_StructureFunctions}

We begin by observing that for $x,\ell \in [0,1]$, making the elementary change of variables $x \mapsto x - \ell/2$, 
the structure functions $S_p(\ell)$ can be described in terms of the increment function   
\[ I(\ell,x)=R(x+\ell/2) - R(x-\ell/2) = 2 \, i \,\sum_{k \geq 1}{\frac{sin(\pi k\ell) }{k} \, \sigma_k\, e_k(x)}, \]
so that if we write $I(\ell) = I(\ell,\cdot)$, we have
\[ S_p(\ell) = \lVert I(\ell) \rVert_p^p.   \]

\begin{lem} \label{LemmaS2}
For $0 < \ell < 1/2$, 
\[ S_2(\ell) \simeq \ell^{3/2}. \]
\end{lem}

\begin{proof}
By Parseval's theorem,
$$
S_2(\ell) \simeq  \sum_{k \geq 1}{ \frac{\sin^2(\pi k\ell)}{k^2}\, \sigma_k^2 }=\underbrace{\sum_{n \leq (2\ell)^{-1/2}}{\frac{\sin^2(\pi n^2 \ell)}{n^4}}}_{A_2(\ell)}+\underbrace{\sum_{n > (2\ell)^{-1/2}}{\frac{\sin^2(\pi n^2 \ell)}{n^4}}}_{B_2(\ell)}.
$$
Since $\sin(x)/x \simeq 1$ for $|x| \leq \pi/2$, we get
$$
A_2(\ell) \simeq \sum_{n \leq (2\ell)^{-1/2}}{\frac{(\pi\,n^2\,\ell)^2}{n^4}} \simeq \ell^{3/2},
$$
while for the second term we have the upper bound
$$
B_2(\ell) \leq \sum_{n > (2\ell)^{-1/2}}{\frac{1}{n^4}} \simeq \int_{(2\ell)^{-1/2}}^{\infty} \frac{\mathrm dx}{x^4} \simeq  \ell^{3/2}.
$$
\end{proof}

\begin{lem} \label{LemmaS4}
There exists $0 < \ell_0 < 1/2$ such that 
\[ S_4(\ell) \simeq \ell^{3} \log(\ell^{-1}), \qquad \forall  \ell \in (0, \ell_0). \] 
\end{lem}

\begin{proof}
We decompose $I(\ell)$ in low and high frequencies so that by the triangle inequality we get
$$
S_4^{1/4}(\ell) = \lVert I(\ell) \rVert_4 = \lVert I_{\leq \ell^{-1}/2}(\ell)+I_{>\ell^{-1}/2}(\ell) \rVert_4 \leq \lVert I_{\leq \ell^{-1}/2}(\ell) \rVert_4 +  \lVert I_{> \ell^{-1}/2}(\ell) \rVert_4.
$$
On the other hand, the Fourier coefficients of $I(\ell)$ are positive for every $1 \leq k \leq \ell^{-1}$, so by Parseval's theorem we deduce
\[
\lVert I_{\leq \ell^{-1}/2}(\ell)  \rVert_4^4 = \lVert I_{\leq \ell^{-1}/2}(\ell)^2 \rVert_2^2 \leq \lVert (I^2)_{\leq \ell^{-1}}(\ell) \rVert_2^2  \leq \lVert I(\ell)^2 \rVert_2^2 = \lVert I(\ell) \rVert_4^4 = S_4(\ell).
\]
In short, we have
\[  \lVert I_{\leq \ell^{-1}/2} (\ell) \rVert_4 \leq  S_4^{1/4}(\ell) \leq  \lVert I_{\leq \ell^{-1}/2}(\ell) \rVert_4 +  \lVert I_{> \ell^{-1}/2}(\ell) \rVert_4,  \]
so we look for both upper and lower estimates for $\lVert I_{\leq \ell^{-1}/2} (\ell) \rVert_4$, but an upper bound suffices for $\lVert I_{> \ell^{-1}/2}(\ell) \rVert_4$.

By Parseval's theorem, we have
$$
\lVert I_{\leq \ell^{-1}/2} (\ell)^2\rVert^2_2=\sum_{k=1}^{\ell^{-1}}{b_k^2}, \qquad \qquad 
b_k=\sum_m{\frac{\sin(\pi m\ell) \sin (\pi (k-m)\ell)}{m(k-m)} \sigma_m \sigma_{k-m}}.
$$
where the index $m$ satisfies $1 \leq m \leq \ell^{-1}/2$ and $1 \leq  k-m \leq \ell^{-1}/2$. 
As above, since $1/2 \leq \sin(x)/x \leq 1$ for $|x| \leq \pi/2$, we get 
$$
b_k \simeq \ell^2\, \sum_m{ \sigma_m \sigma_{k-m}}.
$$
Consequently, with the same restrictions on $m$ as above,
\[
\lVert I_{\leq \ell^{-1}/2}(\ell)^2 \rVert^2_2 \simeq \ell^4 \, \sum_{k=1}^{\ell^{-1}} \left( \sum_m{ \sigma_m \sigma_{k-m}} \right)^2 = \ell^4\, \left\lVert \left(\sum_{k=1}^{\ell^{-1}/2}\sigma_k e_k \right)^2 \right\rVert_2^2 = \ell^4\, \left\lVert \sum_{k=1}^{\sqrt{\ell^{-1}/2}} e_{k^2}  \right\rVert_4^4
\]
By Zalcwasser's theorem in Appendix~\ref{APPENDIX_Zalcwasser}, for $\ell$ small enough we get
$$
\lVert I_{\leq \ell^{-1}/2}(\ell) \rVert^4_4 = \lVert I_{\leq \ell^{-1}/2}(\ell)^2 \rVert^2_2 \simeq \ell^4 \, \ell^{-1} \, \log(\sqrt{\ell^{-1}}) \simeq \ell^3\,  \log(\ell^{-1}).
$$
For the high frequency piece, using again Parseval's theorem, we can write
$$
\lVert I_{> \ell^{-1}/2}(\ell)^2 \rVert^2_2 = \sum_{k >l^{-1}}{\beta_k^2}, \qquad \text{where} \qquad 
\beta_k=\sum_{m,\ k-m > \ell^{-1}/2}{\frac{\sin(\pi m\ell) \sin (\pi (k-m)\ell)}{m(k-m)} \sigma_m \sigma_{k-m}}.
$$
Bounding the sine trivially and using \eqref{NumberTheory} with $N = \ell^{-1}/2$, we get 
\[  \lVert I_{> \ell^{-1}/2}(\ell)^2 \rVert^2_2 \leq  \sum_{k >l^{-1}} \left( \sum_{m,\ k-m > \ell^{-1}/2}{\frac{\sigma_m \sigma_{k-m}}{m(k-m)} } \right)^2 = \lVert (R_{\geq \ell^{-1}/2})^2 \rVert_2^2 = \lVert R_{\geq \ell^{-1}/2} \rVert_4^4 , \]
so by Lemma~\ref{LemmaL4UpperBound} we obtain that
\[   \lVert I_{> \ell^{-1}/2}(\ell) \rVert^4_4  \lesssim (\ell^{-1})^{-3}\, \log{\ell^{-1}} \simeq \ell^3\,\log (\ell^{-1}). \]
\end{proof}
\begin{rmk}
By using Parseval's theorem as in the proof above, one can also compute the asymptotic behaviour of $S_{2k}$ for every $k\in\mathbb{N}$. However, in view of Zalcwasser's Theorem~\ref{APPENDIX_Zalcwasser}, we expect to have a logarithmic term only for the critical $2 k =4$. For $2k > 4$, one expects to get the pure power law $S_{2k} = \ell^{1+k}$ predicted by Jaffard \cite{Jaffard} in \eqref{ExponentJaffard}.
\end{rmk}

\begin{proof}[Proof of Theorem~\ref{Thm1} (Part 2)]
By Lemmas~\ref{LemmaS2} and \ref{LemmaS4},  for small enough $\ell$ we get
\[ G_R(\ell) \simeq \frac{\ell^3\,\log (\ell^{-1})}{(\ell^{3/2})^2} = \log (\ell^{-1}).   \]
\end{proof}

\appendix


\section{The Littlewood-Paley decomposition}\label{APPENDIX_LittlewoodPaley}

We recall here the following classical result \cite[Theorem 3]{LittlewoodPaley}.
\begin{thm}\label{LittlewoodPaleyTheoremAppendix}
Let $p>1$, $A>1$ and  $ f(x) = \sum_{n \in \mathbb{Z}}{ a_n e^{2\pi i n x} } $ a function in $L^p(0,1)$. Consider the decomposition
\[ f(x) = \sum_{k=1}^{\infty}{\Delta_kf(x)} \]
such that
\[ \Delta_1f(x) = \sum_{|n| \leq A}{a_ne^{2\pi i n x}} ,  \qquad  \qquad  
\Delta_kf(x) = \sum_{ A^k < |n| \leq A^{k+1} }{ a_ne^{2\pi i n x} }, \qquad k \geq 2. \]
Then, there exist constants $B_1,B_2 > 0$ depending on $p$ such that 
\[ B_1 \leq \frac{ \lVert \left( \sum_{k=1}^{\infty}{|\Delta_kf|^2} \right)^{1/2} \rVert_{L^p} }{ \lVert f \rVert_{L^p} }  \leq B_2 .\]
\end{thm}

\section{A theorem of Zalcwasser}\label{APPENDIX_Zalcwasser}

The following result from \cite[Equation (57)]{Zalcwasser} is crucial to our proof.
\begin{thm}\label{ZalcwasserTheoremAppendix}
Let $p > 0$. Then, there exist $M_p > 1$ and constants $C_p > c_p>0$ such that for every $N > M_p$, 
\begin{equation}\label{ZalcwasserEstimate}
c_p\,\psi_p(N) \leq \int_0^1{\left| \sum_{m=1}^{N}{e^{2\pi i m^2 x}} \right|^p\,\mathrm dx } \leq C_p\,\psi_p(N)
\end{equation}
where
\begin{equation}\label{DEFINITION_Psi}
\psi_p(N) = 
\begin{cases}
N^{p/2}, & p < 4, \\
N^2\log{N}, & p = 4 \\
N^{p-2}, & p > 4.
\end{cases}
\end{equation}
\end{thm}

\section*{Acknowledgments}
The authors would like to thank Valeria Banica and Luis Vega, who made this collaboration possible through funding and advice, and Francesco Fanelli, Evelyne Miot and Dario Vincenzi, for several useful discussions.

AB and VVDR have been supported by the IUF grant of Valeria Banica. Moreover, they would like to acknowledge the support of ANR ISDEEC. 
DE has been supported by the Ministry of Education, Culture and Sport (Spain) under grant FPU15/03078, as well as by Simons Foundation Collaboration Grant on Wave Turbulence (Nahmod's Award ID 651469). 
DE and VVDR have been supported by the ERCEA under the Advanced Grant 2014 669689 - HADE 
and also by the Basque Government through the BERC 2018-2021 program and by the Ministry of Science, 
Innovation and Universities: BCAM Severo Ochoa accreditation SEV-2017-0718.
VVDR is also supported by NSF grant DMS 1800241.

\bibliographystyle{acm}
\bibliography{BEV_Bibliography}

\end{document}